\documentclass[amssymb,amsfonts]{amsart}
\usepackage{amssymb, latexsym}
\usepackage[T1]{fontenc}
\usepackage[intlimits,sumlimits,noDcommand,narrowiints]{kpfonts}
\usepackage[english]{babel}
\usepackage{microtype}
\usepackage{enumerate}
\usepackage[pdfusetitle]{hyperref}

\theoremstyle{plain}
\newtheorem{thm}{Theorem}
\newtheorem{prop}[thm]{Proposition}

\theoremstyle{definition}
\newtheorem{defn}[thm]{Definition}

\theoremstyle{remark}
\newtheorem{rmk}[thm]{Remark}

\begingroup
\catcode`\%=12
\catcode`\$=12

\gdef\GitPageFooter{$Format: v:\texttt{
\endgroup

\begin{document}
\title[Singularities of axi-symmetric membranes]{Singularities of axially symmetric time-like minimal submanifolds in Minkowski space}
\author[WWY Wong]{Willie Wai Yeung Wong}
\address{Michigan State University}
\email{wongwwy@member.ams.org}

\begin{abstract}
We prove that there does not exist global-in-time axisymmetric solutions to the time-like minimal submanifold system in Minkowski space. We further analyze the limiting geometry as the maximal time of existence is approached. 
\end{abstract}

\maketitle

\section{Introduction}

The equations describing time-like minimal submanifolds of Minkowski space form hyperbolic systems of partial differential equations \cite{Milbre2008, Lindbl2004, Brendl2002, AurChr1979}. Such manifolds describe classical, relativistic evolution of extended test objects \cite{Hoppe2013}, with applications to high energy physics \cite{Kibble1976, VilShe1994} and the study of singularities of semilnear wave equations \cite{Neu1990, Jerrar2011}. 
From the evolutionary point of view of the associated Cauchy problem, it is of interest to classify and differentiate initial data sets exhibiting (future) global-in-time solutions from those that blow-up in finite time. 

Trivially one observes that that any minimal immersion $\Sigma$ into $\mathbf{R}^N$ gives rise to a immersed time-like minimal submanifold $M = \mathbf{R}\times \Sigma$ in $\mathbf{R}^{1,N}$, which exists globally in time. For noncompact spatial-sections, the case where $\Sigma$ is a small perturbation of a linear subspace $\mathbf{R}^d\subsetneq \mathbf{R}^N$ has been previously studied by Brendle in the case $d \geq 3$ and $N = d+1$ \cite{Brendl2002}, Lindblad in the case $d \geq 1$ and $N = d+1$ \cite{Lindbl2004} (see also \cite{Wong2016p} in the case $d = 1$), and Allen, Andersson, and Isenberg in general codimension \cite{AlAnIs2006}. 
They found in this case the perturbed solutions exist globally in time. The general machinery used in those works is the formulation of the evolution equation in the \emph{graphical guage} (see \eqref{eq:graphical} below) and applying the theory of small-data global wellposedness for quasilinear wave equations with null conditions. 
More recently, the stability of $\Sigma$ being the catenoid in $\mathbf{R}^3$ was studied in the axi-symmetric class \cite{DoKrSW2016}. There it was found, in agreement with its variational instability \cite{FisSch1980} that for generic perturbations the solutions diverge exponentially from the catenoid, but also that there exists a co-dimension 1, Lipschitz, center-stable-manifold of initial data in a neighborhood of the catenoid initial data that give rise to global forward-in-time solutions which converge back to the catenoid. In that work the singularity formation is not explored, and it was speculated based on numerical evidence that for perturbations that shrink the ``throat'' of the catenoid, the instability eventually leads to ``the collapse of the throat''. 

In the compact case it is easy to see that for spherically symmetric initial data $\mathbf{S}^d$ immersed in a $\mathbf{R}^{d+1}$ subspace of $\mathbf{R}^N$, that the evolution equations reduce to ordinary differential equations that must blow-up in finite time through a convexity argument. For initial data that is not homogeneous, the only previous works are in the case of strings, where $\Sigma$ has dimension 1 (and hence is a immersion of $\mathbf{S}^1$). Nguyen and Tian studied the case of strings in $\mathbf{R}^{1,2}$ and showed that all initial data leads to finite-time singularity formation \cite{NguTia2013}. 
Jerrard, Novaga, and Orlandi studied the case of higher codimensions \cite{JeNoOr2015} and found that: (a) in higher dimensions there exist both initial data that lead to singularity formation in finite time, and those that lead to global-in-time existence; (b) in $\mathbf{R}^{1,3}$ singularity formation is a stable phenonmenon; (c) in $\mathbf{R}^{1,N}$ with $N \geq 4$ singularity formation is non-generic. 
In both \cite{NguTia2013} and \cite{JeNoOr2015}, the starting point is the observation that the equation of motion of cosmic strings are completely integrable, and in a suitable gauge reduces to the linear wave equation. Therefore solutions necessarily exists, for all time, as a smooth mapping (assuming smooth initial data) of $\mathbf{R}\times\mathbf{S}^1 \to \mathbf{R}^{1,N}$, and it only sufficies to check for the failure of the map to be an \emph{immersion}. 

It may be worthwhile to compare the above results with the well-known statement that, for initial data a closed immersed submanifold (arbitrary codimension) of Euclidean space, the \emph{mean curvature flow} extinguishes in finite time \cite{Smoczy2012}. In broad strokes, one may attribute this difference to the availability of the maximum principle in parabolic (mean curvature flow) but not hyperbolic (time-like minimal surface) equations. Furthermore, by taking Carteisian products, the results of \cite{JeNoOr2015} produce classes of compact, higher dimensional initial data whose evolution under the time-like minimal surface equation exists for all time. Note however, all these examples have codimension at least 2. It is unknown whether there are compact \emph{hypersurfaces} $\Sigma$ of $\mathbf{R}^{N}$ giving rise to a global-in-time time-like minimal hypersurface in $\mathbf{R}^{1,N}$; it is tempting to conjecture in the negative.

The main result of this article is 
\begin{thm}[Rough Statement] \label{thm:rough}
	Given axially symmetric compact initial data of arbitrary codimension, its evolution under the time-like minimal surface equation in Minkowski space forms singularities in finite time, in at least one of the future or the past, due to failure of immersivity.
\end{thm} 
While the result provides a first step toward understanding higher dimensional behavior of spatially-compact time-like minimal immersions, the symmetry assumptions imposed make it much less general compared to the works in $d = 1$ \cite{NguTia2013, JeNoOr2015}.
Nevertheless it is interesting to not only produce a large class of blow-up solutions, but also understand the blow-up mechanism (which is similar to what happens in $d = 1$).  

From the point of view of partial differential equations indeed, the most interesting aspect of the result is not so much that singularities form in finite time, but that we have a breakdown criterion asserting that blow-up reflects failure of immersivity. 
As is well-known (see also the formulation \eqref{eq:graphical} and \eqref{eq:harmmap}) the equations of motion governing time-like minimal immersions can be written in the form of a quasidiagonal system of quasilinear wave equations, the coeffients of which principal part depends on the \emph{first derivatives} of the unknown. The classical breakdown criterion for such wave equations requires controlling the coefficients in $W^{1,\infty}(\Sigma)$, or, in other words, controlling up to two derivatives of the unknown in $L^\infty$. 
It is this fact (together with Sobolev embedding) that governs the classical local-wellposedness regularity threshold placing initial data in regularity class $H^s$ with $s > \frac{d+4}{2}$, where $d$ is the dimension of the spatial slice $\Sigma$. Furthermore, it is known that for generic quasilinear wave equations, shock formation (here it can be understood as blow-up of the second derivatives of the unknown when the first and zero derivatives remain bounded) is a possible source of singularities \cite{Alinha2001, Alinha2001a, Christ2007a, SpHoLW2016}; genericity is understood in the sense of \emph{failing} the null conditions \cite{Alinha2001, Klaine1984a, Klaine1986, Christ1986}. 
It turns out, in the axially symmetric case, there exists an appropriate choice of gauge (analogous to the gauge used in \cite{NguTia2013, JeNoOr2015} to express the equations of motion as linear wave equations) in which the equation vastly simplifies. The availability of this gauge is related to the satisfaction of the null conditions by \eqref{eq:graphical}, but also depends strongly on the axial-symmetry ansatz. (For comparison, symmetry reductions of generic quasilinear wave equations exhibit also shock formation \cite{John1974,John1985}.) In this gauge, the quasilinear coefficients depend only on the unknown, but not its derivatives, and therefore we have available the breakdown criterion at one lower order of regularity. More importantly, however, the gauge also comes equipped with \emph{a priori} $L^\infty$ control on the derivatives of the unknowns, \emph{provided that the quasilinear metric coefficients do not degenerate}. From this we obtain the restriction on the possible blow-up mechanisms. The proof of the necessity of finite-time blow-up then relies on a convexity argument, similar to what is used in the spherical case mentioned previously.  

In sections \ref{s:WMcnx} and \ref{s:conserve} we recall some basic constructions; most of the material are previously known though the presentation and interpretation may be different. The main new contributions to the subject are found in section \ref{s:axialsym} where we define the axial symmetry ansatz and derive some basic properties of the reduced equations, and in section \ref{s:breakdown} where we state and prove the main theorem. The focus of this paper is on the case of spatially compact solutions; in section \ref{s:noncompact} we briefly foray into the noncompact case. 

\section{Connection to wave maps}
\label{s:WMcnx}
Our analysis starts with the following global formulation of the time-like minimal hypersurface equation. Let $M$ be a $(1+d)$-dimensional manifold, and $\phi: M \to \mathbf{R}^{1,N}$ an immersion such that the pull-back metric $g = \phi^*\eta$ ($\eta$ being the Minkowski metric) is Lorentzian. Let $x^\mu|_{\mu = 0, \ldots, N}$ denote the standard rectangular coordinates of $\mathbf{R}^{1,N}$. In \cite{AurChr1979} it was shown that $\phi$ is a minimal immersion if and only if the equation 
\begin{equation}\label{eq:harmmap}
	\Box_{g} (x^\mu\circ\phi) = 0
\end{equation} 
is satisfied for every $\mu$. We remark here that this is in fact a manifestation of the following fact, well-known in the Riemannian case in the harmonic map literature. 

\begin{prop}
	Let $M$ and $P$ be smooth manifolds, where $P$ is equipped with a pseudo-Riemannian metric $h$. Suppose $\phi:M \to P$ is an immersion with nondegenerate pull-back metric $g = \phi^* h$. Then the following are equivalent:
	\begin{enumerate}[(i)]
		\item $\phi$ is minimal;
		\item $\phi$ is a harmonic map $(M,g) \to (P,h)$. 
	\end{enumerate}
\end{prop}
\begin{proof}[Sketch]
	Let $D$ denote the Levi-Civita connections of $P$ and $\nabla$ that of $M$ (with respect to $g$), the (vector) second fundamental form of the immersion $\phi$ can be written as
	\[ k(X,Y) = D_{\phi_* X} \phi_* Y - \phi_*(\nabla_X Y).\]
	Expanding in normal coordinates and taking the $g$-trace immediately gives us the result.
\end{proof}

The equation \eqref{eq:harmmap} is not well-posed as a hyperbolic system of partial differential equations. This is due to diffeomorphism invariance: if $\psi:M \to M$ is a diffeomorphism, then $\phi\circ \psi$ is another solution to \eqref{eq:harmmap}. So if we require $\psi$ to fix a spatial hypersurface and its first jet, we get immediately that the solution to the initial value problem is non-unique. 

The solution to this is well-known: one must fix a guage. A typical local gauge used is the \emph{graphical gauge}. As $\phi$ is a time-like immersion, locally, up to a permutation of the coordinate variables, we can assume $M$ is parametrized by $x^0, \ldots, x^d$, and write $y := (x^{d+1}, \ldots, x^N)$ as a function thereof. Then we recover the usual, local coordinate presentation of the time-like minimal hypersurface equation as the quasilinear system over $(1+d)$-dimensional Minkowski space
\begin{equation}\label{eq:graphical}
\frac{\partial}{\partial x^a} \frac{ \eta^{ab} \partial_b y}{\sqrt{ 1 + \eta^{cd} \langle\partial_c y, \partial_d y\rangle_\eta}} = 0.
\end{equation}
The standard theory of geometric wave equations (localizing to small, globally hyperbolic, coordinate neighborhoods where \eqref{eq:graphical} applies and then gluing by local uniqueness) gives well-posedness of the corresponding initial value problem for data in the Sobolev space $H^s$, $s > \frac{d}{2} + 2$; see \cite{Milbre2008} for a discussion of the local initial value problem in the coordinate formulation. 

We remark here that as the coordinate functions $x^0, \ldots, x^d$ themselves solve the wave equation \eqref{eq:harmmap}, the formulation \eqref{eq:graphical} is in fact writing the equation of motion in local harmonic/wave coordinates. Similar to the situation in general relativity (see Theorem 7.4 of \cite{Choque2009}), the local well-posedness theory can also be studied in a wavemap gauge to avoid the spatial localization step. 

\section{Co-moving gauge and conservation law}
\label{s:conserve}
For the axially symmetric set-up we are interested in, it is however more convenient to study the equation in a co-moving gauge. We endow the manifold $M$ with the following coordinate system: let $t = x^0(\phi)$ be the background time-coordinate function. Wherever $\phi$ is well-defined as an immersion, we can associate to $t$ the vector field $\tau$ such that $\tau(t) = 1$ and $\tau$ is orthogonal to the level sets of $t$. The flow of $\tau$ induces diffeomorphisms between level sets of $t$, which we denote by $\Sigma_t$, and so we can identify $M$, at least locally in time (and ``globally'' in space) with $(-T,T) \times \Sigma_0$. 

By construction, we have
\begin{equation}\label{eq:tautsharp}
	\mathrm{d}t^\sharp = \frac{\tau}{g(\tau,\tau)}.
\end{equation}
From \eqref{eq:harmmap} we know that $\mathrm{d}t^\sharp$ is a divergence-free vector field, which implies that $\mathrm{d} \iota_{\mathrm{d}t^\sharp} \mathrm{dvol}_g = 0$, where $\mathrm{dvol}_g$ is the spacetime volume form on $M$ given by the metric $g$. Now, let $\sigma$ denote the $d$-form $\frac{1}{\sqrt{|g(\tau,\tau)|}}\iota_\tau \mathrm{dvol}_g$, which is the volume form given by the induced metric on $\Sigma_t$. Our computations above indicate that
\begin{equation}\label{eq:cons}
	\mathcal{L}_\tau \frac{\sigma}{\sqrt{|g(\tau,\tau)|}} = \mathrm{d} \iota_\tau \iota_\tau \frac{\mathrm{dvol}_g}{|g(\tau,\tau)|} - \iota_\tau \mathrm{d} \iota_{\mathrm{d}t^\sharp} \mathrm{dvol}_g = 0.
\end{equation}
We note that the conservation law \eqref{eq:cons} is previously known (e.g.\ equation (2.19) in \cite{Hoppe2013}) and can be interpreted as a relativistic conservation of mass formula.

An immediate consequence is
\begin{prop}\label{prop:parttwo}
	Let $\Sigma$ be a $d$-dimensional manifold. Suppose for $t\in [0,T]$ there exists a continuous family of $C^2$ mappings $\phi_t: \Sigma \to \mathbf{R}^{d+1}$, such that the mapping $(0,T)\times \Sigma \ni (t,q) \mapsto (t,\phi_t(q)) \in \mathbf{R}^{1,d+1}$ is a smooth minimal immersion in the comoving gauge. Then $\phi_T$ is an immersion if and only if $g(\tau,\tau)$ remains bounded away from zero near $t = T$. 
\end{prop}
\begin{proof}
	Since $\phi_T$ is an immersion if and only if $\mathrm{d}\phi_T$ is full rank, which is equivalent to the induced volume form $\sigma$ being non-degenerate, the claim then follows from the conservation law \eqref{eq:cons}. 
\end{proof}

Now take $U$ a small coordinate neighborhood on $\Sigma_0$, parametrized by $(y^1, \ldots, y^d)$. Using the flow of the vector field $\tau$ we extend this to a local coordinate system for $M$ with coordinate functions $(t= y^0, y^1, \ldots, y^d)$; in this coordinate system observe that $\tau = \partial_t$. By construction we require our immersion $\phi: M\to \mathbf{R}^{1,d+1}$ satisfy 
\[ t = x^0\circ \phi(t,y^1, \ldots, y^d).\]
And the induced metric on $M$ satisfies
\begin{equation}\left\{
\begin{aligned}
	g_{ij} &= \left\langle \frac{\partial}{\partial y^i} \phi, \frac{\partial}{\partial y^j} \phi \right\rangle_\eta, \quad i,j\in \{1, \ldots, d\};\\
	g_{0j} &= 0, \quad j \in \{1, \ldots, d\}; \\
	g_{00} &= g(\tau,\tau) = \left\langle \frac{\partial}{\partial t} \phi, \frac{\partial}{\partial t}\phi\right\rangle_\eta.
\end{aligned}\right.
\end{equation}
In this coordinate system then observe that 
\[ \sigma = \sqrt{\frac{|\det g|}{|g_{00}|}} \mathrm{d}y^1\wedge \cdots\wedge \mathrm{d}y^d.\]
Furthermore, by orthogonality the inverse metric 
\[ g^{00} = \frac{1}{g_{00}}.\]
So a direct consequence of the conservation law \eqref{eq:cons} is that the wave equation \eqref{eq:harmmap} satisfied by the coordinate components $x^\mu\circ\phi$ takes the form
\begin{equation}\label{eq:lbform}
	\partial^2_{tt} (x^\mu\circ \phi) = \frac{|g_{00}|}{\sqrt{|\det g|}} \sum_{i,j = 1}^d \frac{\partial}{\partial y^i} \left( \sqrt{|\det g|} g^{ij} \frac{\partial}{\partial y^j} (x^\mu \circ \phi) \right).
\end{equation}

\section{Axial symmetry}
\label{s:axialsym}

For the purposes of this paper, we take the following definition of axial symmetry. 
\begin{defn}
	Given a $d$-dimensional closed immersed submanifold $\Sigma\subset \mathbf{R}^N$, we say it is \emph{axially symmetric} if 
	\begin{enumerate}[(i)]
		\item there exists some $k\geq 1$, a set of numbers $d_1, \ldots, d_k \geq 1$ with $\sum_{i = 1}^k d_i = d -1$, and a parametrization $\Psi: \mathbf{S}^1 \times \mathbf{S}^{d_1}\times \cdots \times \mathbf{S}^{d_k} \to \Sigma$, such that
		\item for $(y, \theta_1, \ldots, \theta_k)\in  \mathbf{S}^1 \times \mathbf{S}^{d_1}\times \cdots \times \mathbf{S}^{d_k}$, the mapping 
		\begin{multline*} \Psi(y, \theta_1, \ldots, \theta_k) = \\(z(y), r_1(y) \theta_1, r_2(y)\theta_2, \ldots, r_k(y)\theta_k) \in \mathbf{R}^{N - d + 1 - k}\times \mathbf{R}^{d_1 + 1} \times \cdots \times \mathbf{R}^{d_k + 1} \cong \mathbf{R}^N \end{multline*}
			for some functions $z:\mathbf{S}^1 \to \mathbf{R}^{N-d+1-k}$ and $r_i: \mathbf{S}^1 \to \mathbf{R}_+$. 
	\end{enumerate}
\end{defn}
The definition above is a generalization of the classical notion of surfaces of revolution, which corresponds to $k = d_1 = 1$ and $N = 3$. 
Borrowing from the classical terminology we also refer to the function $(z, r_1, \ldots, r_k): \mathbf{S}^1 \to \mathbf{R}^{N-d+1-k}\times (\mathbf{R}_+)^k$ as ``the generatrix'' of $\Sigma$. Notice that we have the freedom to reparametrize $y$. Adding a dependence on the time-parameter $t$ to the generatrix, the above definition also generalizes in the obvious way to $d+1$-dimensional time-like immersed submanifolds of $\mathbf{R}^{1,N}$. Is is also obvious from the construction that the definition is compatible with the co-moving gauge condition of the previous section.

\begin{rmk}
	If we take $k = 0$ in the definition above, we see that the resulting object is simply an immersion of $\mathbf{S}^1$ in some Euclidean space, and we recover the case of cosmic strings. Our definition requires $k \geq 1$ simply for convenience of stating the main theorem later: as evident by the analysis of \cite{JeNoOr2015} the ``finite-time collapse'' property doesn't always hold for $k = 0$. 

	On the other hand, if we take $\sum d_i = d$ instead of $d -1$, then we have a product of spheres which collapses in finite time by the convexity argument mentioned in the introduction. 
\end{rmk}

Let $M \cong (t_1, t_2)\times \Sigma$ be an axially symmetric time-like immersed subfmanifold in Minkowski space $\mathbf{R}^{1,N}$. The induced metric on $M$ in co-moving gauge is given in warped-product form by
\begin{equation}\label{eq:metric}
	g = g_{tt} \mathrm{d}t^2 + g_{yy} \mathrm{d}y^2 + \sum_{i = 1}^k (r_i)^2 g_{\mathbf{S}^{d_i}}
\end{equation}
where $g_{\mathbf{S}^{d_i}}$ are the standard metrics of the unit spheres in $\mathbf{R}^{d_i+1}$, and
\begin{subequations}
	\begin{align}
		g_{tt} &= -1 + |\partial_t z|^2 + \sum_{i = 1}^k (\partial_t r_i)^2, \\
		g_{yy} &= |\partial_y z|^2 + \sum_{i = 1}^k (\partial_y r_i)^2.
	\end{align}
\end{subequations}
We remark that the time-like condition is enforced by $g_{tt} < 0$. The corresponding volume form 
\[ \mathrm{dvol}_g = \sqrt{|\det g|}~ \mathrm{d}t \wedge \mathrm{d}y \wedge \mathrm{dvol}_{\mathbf{S}^{d_1}} \wedge \cdots \wedge \mathrm{dvol}_{\mathbf{S}^{d_k}}\]
has coefficient
\begin{equation}\label{eq:detg}
	\sqrt{|\det g|} = \sqrt{|g_{tt}| g_{yy}} \prod_{i = 1}^k (r_i)^{d_i}.
\end{equation}

With the parametrization $y$ of $\mathbf{S}^1$ still taken to be arbitrary, the equations of motion \eqref{eq:lbform} for $M$ being a time-like minimal immersion can be computed to be the following system of equations for the generatrix.
\begin{equation}\label{eq:generalgeneratrix}
	\left\{ \begin{aligned}
		- \partial^2_{tt} z + \frac{|g_{tt}|}{\sqrt{|\det g|}} \partial_y \left( \frac{\sqrt{|\det g|}}{|g_{yy}|} \partial_y z\right) &= 0; \\
		- \partial^2_{tt} r_j + \frac{|g_{tt}|}{\sqrt{|\det g|}} \partial_y \left( \frac{\sqrt{|\det g|}}{|g_{yy}|} \partial_y r_j\right) &= \frac{d_j |g_{tt}|}{r_j}, \quad j \in \{1, \ldots, k\}.
		\end{aligned}\right.
\end{equation}
Using the freedom of reparametrization for $y$, we can fix, for one $t_0\in (t_1,t_2)$, that for some $C > 0$ that
\begin{equation}\label{eq:gaugefix}
	\frac{|g_{tt}|}{\sqrt{|\det g|}}(t_0,y) = C \text{ for all } y\in \mathbf{S}^1.
\end{equation}
Then \eqref{eq:cons} implies the same equality holds true for all $t\in (t_1,t_2)$. 

\begin{rmk}
	The gauge condition \eqref{eq:gaugefix} is the analogue, in our construction, of the ``orthogonal gauge condition'' of \cite{BeHoNO2010} that is used in \cite{NguTia2013} and \cite{JeNoOr2015}. The proof is by integrating the change of variables formula (chain rule) and the constant $C$ is chosen to fix $y$ to be $2\pi$ periodic. Note further that for the Cauchy problem \eqref{eq:gaugefix} can be implemented on the initial data, as the metric $g$ depends only on up-to-one derivatives of the generatrix functions.
\end{rmk}

Under the gauge condition \eqref{eq:gaugefix} together with \eqref{eq:detg}, the equations of motion \eqref{eq:generalgeneratrix} can be further reduced to 
\begin{equation}\label{eq:axisymsys}
	\left\{ \begin{aligned}
			- \partial^2_{tt} z + C^2\partial_y \left( r_1^{2d_1}r_2^{2d_2} \cdots r_k^{2d_k} \partial_y z\right) &= 0; \\
		- \partial^2_{tt} r_j +  C^2\partial_y \left( r_1^{2d_1}r_2^{2d_2} \cdots r_k^{2d_k} \partial_y r_j\right) &= \frac{d_j |g_{tt}|}{r_j}, \quad j \in \{1, \ldots, k\}.
		\end{aligned}\right.
\end{equation}
The main beneficial feature of \eqref{eq:axisymsys} is that the coefficients of the principal part of the the equations no longer depend on the derivatives of the generatrix functions, and that the equations are hyperbolic so long as the product $r_1\cdots r_k \neq 0$. 

\begin{prop}[Preservation of gauge]\label{prop:gaugepre}
	Let $(z, r_1, \ldots, r_k)$ be a smooth solution to \eqref{eq:axisymsys} with the product $r_1\cdots r_k > 0$, such that at some initial time $t_0$ the gauge condition \eqref{eq:gaugefix} as well as the orthogonality (co-moving) gauge condition 
	\begin{equation}\label{eq:comoving}
		0 = g_{ty} := \langle \partial_t z, \partial_y z\rangle + \sum_{i = 1}^k \langle \partial_t r, \partial_y r\rangle
	\end{equation}
	are satisfied. Then both \eqref{eq:gaugefix} and \eqref{eq:comoving} are satisfied for all time. 
\end{prop}
\begin{proof}
	Write $R = r_1^{2d_1} \cdots r_k^{2d_k}$. Let $X = R^{-1} g_{tt} + C^2 g_{yy}$, and let $Y = g_{ty}$. 
	That \eqref{eq:gaugefix} and \eqref{eq:comoving} are satisfied at an initial time $t_0$ means $X(t_0, \cdot) = 0 = Y(t_0,\cdot)$. A direct computation using \eqref{eq:axisymsys} shows that $X$ and $Y$ solves the following hyperbolic system of partial differential equations:
	\[ \left\{
	\begin{aligned}
		\partial_t Y &= \frac12 R^{-1} \partial_y (R^2 X); \\
		\partial_t X &= 2 C^2 R^{-1} \partial_y (R Y) . 
	\end{aligned} \right.\]
	In particular this implies the energy identity 
	\[ \partial_t \int_{\mathbf{S}^1} Y^2 + \frac{1}{4C^2} R X^2 ~\mathrm{d}y = \int_{\mathbf{S}^1} \frac{1}{4C^2} (\partial_t R)X^2 + (\partial_y R) XY~\mathrm{d}y .\]
	So by Gronwall's lemma and the assumption that $R > 0$ and $(z, r_1, \ldots, r_k)$ is smooth, we have that the gauge conditions are propagated. 
\end{proof}

The above proposition implies that solutions of the reduced equations \eqref{eq:axisymsys} with the gauge conditions \eqref{eq:gaugefix} and \eqref{eq:comoving} satisfied initially necessarily gives rise to axially symmetric time-like immersed minimal submanifolds, provided that the corresponding metric \eqref{eq:metric} is non-degenerate and timelike. 

\section{The Main Theorem}
\label{s:breakdown}

The following breakdown criterion is standard with a proof using the energy method and Sobolev interpolation; see e.g. \cite[Chapter 6]{Horman1997}. 
\begin{prop}\label{prop:stdbd}
	Suppose $(z, r_1, \ldots, r_k)$ is a  smooth solution to the system \eqref{eq:axisymsys} on $D = (t_1, t_2)\times \mathbf{S}^1$, such that $\inf_{D} r_1r_2\cdots r_k > 0$, and $\sup_{D} |\partial^\alpha z,r| < \infty$ for all multi-indices $\alpha$ satisfying $|\alpha| \leq 1$.
	Then there exists some $\epsilon > 0$ such that $(z, r_1, \ldots, r_k)$ extends to a smooth solution of \eqref{eq:axisymsys} on $(t_1 - \epsilon, t_2 + \epsilon) \times \mathbf{S}^1$. 
\end{prop}

Now let us return to the actual problem of immersed time-like minimal submanifolds. Notice that whenever $M$ is a smooth immersed time-like minimal submanifold of $\mathbf{R}^{1,N}$, the time-like vector field $\tau$ is well-defined, as is the \emph{non-degenerate} spatial-volume form $\sigma$. As a consequence of Proposition \ref{prop:parttwo}, this means that the obstructions to extending $M$ as a smooth immersed time-like minimal submanifold are (a) loss of regularity or (b) loss of immersivity. 

From the discussion in section \ref{s:axialsym}, we see that any spatially-compact \emph{axially symmetric} smooth immersed time-like minimal submanifold in $\mathbf{R}^{1,N}$ can be parametrized in such a way that \eqref{eq:comoving} and \eqref{eq:gaugefix}, as well as the requirement $g_{tt} < 0$, hold, and that in this gauge the generatrix functions satisfy \eqref{eq:axisymsys}. 
Conversely, by Proposition \ref{prop:gaugepre} we see that any smooth generatrix on $(t_1,t_2)\times\mathbf{S}^1$ solving \eqref{eq:axisymsys} and satisfying the condition $g_{tt} < 0$, and, for some fixed time $t_0$, verify \eqref{eq:comoving} and \eqref{eq:gaugefix}, represents a spatially-compact axially symmetric smooth immersed time-like minimal submanifold. 
Therefore we will identify the two. 
\begin{defn}
	For the remainder of this section, an \emph{axi-symmetric smooth membrane} $M$ is a smooth solution $(z,r_1, \ldots, r_k)$ of \eqref{eq:axisymsys} on $(t_1,t_2)\times \mathbf{S}^1$ such that the conditions \eqref{eq:comoving}, \eqref{eq:gaugefix}, and $g_{tt} < 0$ are satisfied. Given another axi-symmetric smooth membrane $M'$ with associated times $t_1', t_2'$ as well as solution $(z', r_1', \ldots, r_k')$, we say that $M'$ is an \emph{extension} of $M$ if $t_1' \leq t_1 < t_2 \leq t_2'$ and $(z', r_1', \ldots, r_k')|_{(t_1,t_2)} = (z, r_1, \ldots, r_k)$. The extension is \emph{strict} if at least one of $t_1' < t_1$ or $t_2 < t_2'$ is true. 
	Finally, an axi-symmetric smooth membrane $M$ is said to be \emph{maximal} if it does not admit any strict extensions.
\end{defn}
By the uniqueness of smooth solutions to nonlinear wave equations, immediately we see that any axi-symmetric smooth membrane $M$ admits a unique, maximal extension $\tilde{M}$. Our main result is the following.
\begin{thm}[Main Theorem] \label{thm:main}
	The maximal extension $\tilde{M}$ of any axi-symmetric smooth membrane $M$ terminates in finite time. That is to say, the time domain $(T_1, T_2)$ of $\tilde{M}$ is such that at least one of $T_1$ or $T_2$ is finite. Furthermore, we have the following asymptotic control as $t\nearrow T_2$ and $t\searrow T_1$:

		In the case of finite $T_1$ or $T_2$ it holds that
		\begin{equation}\label{eq:lossim1} 
		\liminf \left[ \inf_{y} |g_{tt}| r_1\cdots r_k \right](t) = 0;
		\end{equation}

		In the case of infinite $T_1$ or $T_2$ it holds that 
			\begin{equation}\label{eq:lossim2} 
				\liminf \int_{\mathbf{S}^1} |g_{tt}|(t,y) ~\mathrm{d}y =  0.
			\end{equation}
\end{thm}
\begin{proof}
	We first prove the extension criterion: if $\inf_M |g_{tt}| r_1\cdots r_k > 0$, and if $|t_1|, |t_2| < \infty$, then $M$ is not maximal. Note that this implies \eqref{eq:lossim1}. Observe by assumption that $1 + g_{tt} \geq 0$ and our assumption that $M$ satisfies $g_{tt} < 0$ implies that $|\partial_t z|^2 + \sum |\partial_t r_i|^2 < 1$. This further implies that $r_i \leq r_i(t_0) + |t_2 - t_1|$ and $|z| \leq |z(t_0)| + |t_2 - t_1|$ for any $t_0 \in (t_1,t_2)$. Hence we conclude 
	\[ \inf_M |g_{tt}| > 0.\]
	Next, since $|g_{tt}| \leq 1$ we also have 
	\[ \inf_M r_1 \cdots r_k > 0.\]
	The assumption that \eqref{eq:gaugefix} holds implies that 
	\[ |\partial_y z|^2 + \sum (\partial_y r_i)^2 = g_{yy} = \frac{C^{-2} |g_{tt}|}{(r_1\cdots r_k)^2} \]
	is bounded on $M$. Thus the hypothesis of Proposition \ref{prop:stdbd} is satisfied and $(z, r_1, \ldots, r_k)$ extends as a solution to \eqref{eq:axisymsys}. By Proposition \ref{prop:gaugepre} the conditions \eqref{eq:comoving} and \eqref{eq:gaugefix} still hold for the extended solutions, as they held for $M$. And by continuity of the generatrix and compactness of $\mathbf{S}^1$, since $\inf_M |g_{tt}|r_1\cdots r_k > 0$ it must remain bounded for a sufficiently small extension.

	Next we examine the equations of motion \eqref{eq:axisymsys}. For $j\in \{1, \ldots, k\}$, the equation for $r_j$ reads 
	\[ - \partial^2_{tt} r_j +  C^2\partial_y \left( r_1^{2d_1}r_2^{2d_2} \cdots r_k^{2d_k} \partial_y r_j\right) = \frac{d_j |g_{tt}|}{r_j}.\]
	Integrate in $y$ to define the mean 
	\[ \overline{r_j}(t) := \int_{\mathbf{S}^1} r_j(t,y)~\mathrm{d}y \]
	we see that 
	\begin{equation}\label{eq:convmeanrad} 
	\ddot{\overline{r_j}} = \int_{\mathbf{S}^1} \frac{d_j g_{tt}}{r_j}~\mathrm{d}y < 0.
	\end{equation}
	The convexity condition \eqref{eq:convmeanrad} immediately means that $\overline{r_j}$ is would reach zero in finite time: if $\dot{\overline{r_j}}(t_0) \leq 0$ for some $t_0\in (T_1,T_2)$, then it is monotonic decreasing for $t > T_0$ and thus $T_2 < \infty$; similarly for $T_1$ if $\dot{\overline{r_j}}(t_0) \geq 0$. This shows that $T_1$ and $T_2$ cannot both be infinite. 

	It remains to show \eqref{eq:lossim2} when one of $T_1, T_2$ is infinite. Without loss of generality suppose it is $T_2$. Supposing for contradiction that \eqref{eq:lossim2} doesn't hold, and thus there exists $T_0\in (T_1, \infty)$ and $\epsilon > 0$ such that for every $t\in (T_0,\infty)$
	\[ \int_{\mathbf{S}^1} |g_{tt}|(t,y) ~\mathrm{d}y > \epsilon.\]
	Then recalling that we have the \emph{a priori} bound that $r_j(t,y) \leq r_j(T_0,y) + |t - T_0|$, we obtain from \eqref{eq:convmeanrad} that
	\[ \ddot{\overline{r_j}}(t) = \int_{\mathbf{S}^1} \frac{d_j g_{tt}}{r_j}~\mathrm{d}y \leq \frac{d_j}{|t - T_0| + \sup_{y} r_j(T_0,y)} \int g_{tt} ~\mathrm{d}y \leq \frac{ - d_j \epsilon}{|t - T_0| + \sup_{y} r_j(T_0,y)}.\]
	As $\frac{1}{|t-T_0|}$ is not integrable, this implies that within finite time $\dot{\overline{r_j}} < 0$ and this contradicts the assumption that $T_2 = \infty$. 
\end{proof}

\begin{rmk}
	Notice that as part of the proof, the decay of $|g_{tt}| r_1\cdots r_k$ to zero must hold as we approach any first singular points. This condition in particular implies that the obstruction to further extension is in fact loss of immersivity. On the other hand, we do not rule out the possibility that the loss of immersivity is accompanied by loss of regularity! This should be contrasted with the results of \cite{NguTia2013, JeNoOr2015} where the reduced equations remain always regular and the singularity is purely geometric. One can easily see, however, by Proposition \ref{prop:stdbd} that if $r_1\cdots r_k$ remains bounded away from zero, then the singularity is one with only loss of immersivity: there exists smooth solutions to \eqref{eq:axisymsys} continuing past the singularity. 
\end{rmk}

\section{Noncompact case}
\label{s:noncompact}

In this section we collect some remarks concerning the application of the results of this paper to non-compact submanifolds. The results of sections \ref{s:WMcnx} and \ref{s:conserve} are general and largely remain unchanged. Proposition \ref{prop:parttwo} can be naturally localized in space since the result is essentially pointwise. In section \ref{s:axialsym}, in defining axial symmetric submanifolds, the first factor $\mathbf{S}^1$ parametrized by $y$ now obviously have to be replaced by $\mathbf{R}$. The availability of the gauge \eqref{eq:gaugefix} remains unchanged, except we can take $C\equiv 1$ (by rescaling on $\mathbf{R}$). Using finite speed of propagation for hyperbolic systems, Proposition \ref{prop:gaugepre} also remains true; the proof just requires suitable spatial localization. 

In terms of the main results, after spatially localizing to globally hyperbolic domains an analogue to Proposition \ref{prop:stdbd} also holds in the non-compact case; this is again using finite-speed-propagation property for hyperbolic systems. The main difference occurs in Theorem \ref{thm:main}. Instead of fixing the upper and lower times, one can consider instead the maximal globally hyperbolic extension, with the domain (of the generatrix, in other words the variables $(t,y)$) some open subset $U$ of $\mathbf{R}\times\mathbf{R}$. Suppose $(t_0, y_0)$ is a point in the future boundary portion of $\partial U$. Since $U$ is globally hyperbolic, we can speak of the ``chronological past'' $\mathcal{I}^-(t_0,y_0)$ of $(t_0,y_0)$ as an open subset of $U$ (this is the union of all time-like curves whose upper terminus is $(t_0,y_0)$). We say that $(t_0,y_0)$ is \emph{space-like} if the closure of $\mathcal{I}^-(t_0,y_0)$ is contained in $U$. 

With this definition, spatially localizing the proof of Theorem \ref{thm:main} gives us the following analogue of \eqref{eq:lossim1}:  if $(t_0,y_0)\in \partial U$ is a spacelike future boundary point and if $t_0 < \infty$, then we can conclude $\inf_{\mathcal{I}^{-}(t_0,y_0)} |g_{tt}| r_1\cdots r_k = 0$. The rest of Theorem \ref{thm:main} do not hold, as the catenoid (and certain of its perturbations) provide counterexamples.

\bibliographystyle{amsalpha}
\bibliography{mixmaster.bib}

\end{document}